\documentclass{amsart}

\usepackage{amsmath,amsthm,amssymb}
\usepackage{hyperref}
\usepackage{todonotes}

\newtheorem{thm}{Theorem}[section]
\newtheorem{prop}[thm]{Proposition}
\newtheorem{lemma}[thm]{Lemma}
\newtheorem{cor}[thm]{Corollary}
\newtheorem{rmk}[thm]{Remark}

\newtheorem{openprob}[thm]{Open Problem}
\theoremstyle{definition}
\newtheorem{example}[thm]{Example}
\newtheorem{defn}[thm]{Definition}

\newcommand{\sets}{\mathsf{Set}}
\newcommand{\ords}{\mathsf{On}}
\newcommand{\zf}{\mathbf{ZF}}
\newcommand{\izf}{\mathbf{IZF}}

\newcommand{\powset}{\mathcal{P}}
\newcommand{\sym}{\operatorname{Sym}}

\newcommand{\rk}{\operatorname{rank}}
\newcommand{\im}{\operatorname{im}}
\newcommand{\tc}{\operatorname{TC}}

\newcommand{\node}{\mathtt{node}}

\title{A class of higher inductive types in Zermelo-Fraenkel set
  theory}
\author{Andrew W Swan}
\date{\today}

\begin{document}

\maketitle
\begin{abstract}
  We define a class of higher inductive types that can be constructed
  in the category of sets under the assumptions of Zermelo-Fraenkel set
  theory without the axiom of choice or the existence of uncountable
  regular cardinals. This class includes the example of unordered
  trees of any arity.
\end{abstract}

\section{Introduction}
\label{sec:introduction}

Higher inductive types are one of the key ideas in homotopy type
theory \cite{hottbook}. We think of an (ordinary) inductive type as
the smallest type closed under certain algebraic operations or
\emph{point constructors}. For instance, we define the type of
countably branching trees $T$ to be the smallest type closed under the
following operations.
\begin{align*}
  \mathtt{leaf} &: T \\
  \node &: (\omega \to T) \to T
\end{align*}
Within type theory we formalise the idea that $T$ is the smallest type
with the above point constructors using recursion or induction
terms. However, semantically, it is often more convenient to think in
terms of initial algebras. We say an \emph{algebra} for the above
constructors is a type $X$ together with a map $1 + X^\omega \to
X$. $T$ is then the initial object in the category of algebras. This
is a classic example of a \emph{$W$-type}, as defined by Moerdijk and
Palmgren \cite{moerdijkpalmgrenwtypes}.

For higher inductive types, one not only has point constructors, but
also \emph{path constructors}, which add proofs of identities of
terms. Higher inductive types are usually considered within HoTT and
have well understood semantics within models of HoTT
\cite{lumsdaineshulmanhit}, \cite{chmhitsctt},
\cite{cavalloharper}. However, since they are defined within the language of
type theory, one might also consider whether they hold in
interpretations of extensional type theory in locally cartesian closed
categories, and in particular the category of sets, $\sets$.

One of the simplest examples of higher inductive type is pushouts. In
$\sets$ these can be implemented as pushouts in the usual categorical
sense. It follows that $\sets$ contains all of the $n$-dimensional
spheres, although there is not much you can say about them without the
univalence axiom, and indeed they turn out to be trivial in
$\sets$. Quotients and image factorisations are examples of simple
colimits that play a useful role even within models of extensional
type theory \cite{maiettimodular}, \cite{awodeybauerpat}.

There are also more complicated examples of higher inductive types
that are non trivial in extensional type theory, and even $\sets$,
within the framework of \emph{quotient inductive types}
\cite{altenkirchkaposittintt}. In fact our examples of interest fall
within a smaller class with a simpler definition and clearer
semantics. This class was studied by Blass \cite{blassfreealg} under
the name \emph{free algebras subject to identities} and by Fiore,
Pitts and Steenkamp in \cite{fiorepittssteenkamp} under the name
\emph{$QW$-types} (we will refer to them by the latter name). A well
known example of such a type is that of ``unordered countably
branching trees.'' We modify the definition of $T$ above to get the
higher inductive type $T_{\sym}$ by adding equations as follows,
where we write $\sym(\omega)$ for the type of permutations
$\omega \to \omega$.
\begin{align*}
  \mathtt{leaf} &: T_{\sym} \\
  \node &: (\omega \to T_{\sym}) \to T_{\sym} \\
  \mathtt{perm} &: \prod_{f \colon \omega \to T_{\sym}}
                  \prod_{\pi : \sym(\omega)} \node(f) =
                  \node(f \circ \pi)
\end{align*}
Altenkirch, Capriotti, Dijkstra, Kraus and Forsberg include
this in \cite{acdkfqiits}, as a non trivial example of a quotient
inductive(-inductive) type. As they remark, the obvious construction
of $T_{\sym}$ as a quotient of $T$ requires the axiom of
choice.\footnote{For the special case above, countable choice would be
  enough.} Fiore, Pitts and Steenkamp showed that in fact it is an
example of a $QW$-type \cite[Example 2]{fiorepittssteenkamp}.

Blass showed that all $QW$-types can be constructed in $\sets$ under
the assumption that regular cardinals are unbounded in the class of all
ordinals. More generally free algebras can be constructed in
cocomplete categories from the existence of regular cardinals of
sufficiently high cardinality via the general techniques of Kelly
\cite{kellytransfinite}. For example this plays an important role in
the construction of higher inductive types by Lumsdaine and Shulman
\cite{lumsdaineshulmanhit}. The existence of an unbounded class of
regular ordinals is usually a reasonable one. It follows from very
weak versions of choice, such as $\mathbf{WISC}$ \cite{vdbergwisczf}
and a variant is often assumed in constructive set theory \cite[Section
10]{aczelrathjen}. It is also the case that every inaccessible
cardinal is in particular regular.\footnote{Note, however that even in the
presence of inaccessible cardinals it can be useful to have proofs that
are valid in $\zf$ without further assumption: if $\kappa$ is
inaccessible, then $V_\kappa$ is a transitive model of $\zf$, and so
any proof valid in $\zf$ can be carried out inside it (e.g. to
construct HITs that belong to $V_\kappa$), but $V_\kappa$
itself does not contain inaccessible cardinals without further
assumptions on $\kappa$.}

However, Gitik \cite{gitik1980} has constructed a model of $\zf$ in
which $\omega$ is the only regular cardinal.\footnote{under certain
  large cardinal assumptions} Moreover, Blass showed that the
assumption is strictly necessary, by constructing a $QW$-type which is
isomorphic to the collection of ordinals hereditarily of countable
cofinality, if it exists. He deduced by Gitik's result that this gives
an example of a $QW$-type that does not provably exist in $\sets$
under the assumptions of $\zf$.\footnote{Using Blass' argument and a
  later paper by Gitik \cite{gitik85} one can also show the
  following. For every successor ordinal $\alpha$ there is a transitive
  model of $\zf$ where Blass' example of a $QW$-type exists and is
  precisely the set of all ordinals less than $\omega_\alpha$.}

Fiore, Pitts and Steenkamp in loc. cit. gave an electronically
verified proof that $QW$-types can be constructed in type theory using Agda sized
types and universes closed under inductive-inductive types. We can see
from Blass' counterexample that some combination of these assumptions
for $\sets$ must lead to the existence of uncountable regular
cardinals. In a second paper \cite{fiorepittssteenkamp2} the same
authors showed that the weak choice principle $\mathbf{WISC}$ can be
used in place of Agda sized types, again verifying the result
electronically.

On the other hand, some higher inductive types can be constructed in
$\sets$ without choice or unbounded regular cardinals. In addition to
colimits, as mentioned above, the author showed in \cite{swanwtypered}
that $W$-types with reductions exist in any boolean topos, including
$\sets$. A similar argument shows that Sojakova's notion of
$W$-suspensions \cite{sojakovawsusp} also exist in all boolean
toposes.\footnote{A proof is left as an exercise for the reader.}

In this paper we will see a new class of $QW$-types that can be
constructed in $\sets$ under $\zf$, without any assumptions of choice
or existence of regular cardinals, that we call \emph{image
  preserving} $QW$-types. This will included the example of unordered
countably branching trees above, and more generally unordered trees of
any arity. The proof is based on a construction of a set of all
hereditarily countable sets due to Jech \cite{jechheredctbl}. We will
be able to recover hereditarily countable sets as a special case, and
moreover a later generalisation due to Holmes \cite{holmesheredsmall}.

\subsection*{Acknowledgements}
\label{sec:acknowledgements}

I am grateful to Jonas Frey, Simon Henry and Thomas Streicher for
helpful discussion and suggestions.

I gratefully acknowledge the support of the Air Force Office of
Scientific Research through MURI grant FA9550-15-1-0053. Any opinions,
findings and conclusions or recommendations expressed in this material
are those of the authors and do not necessarily reflect the views of
the AFOSR.

\section{Image preserving $QW$-types}

We now define our class of higher inductive types that we will
construct in $\sets$. It will be clear by the definition that this is a
special case of $QW$-types \cite{fiorepittssteenkamp}.

\begin{defn}
  A $\emph{polynomial}$ is a set $A$ together with a family of sets
  $(B_a)_{a \in A}$. We will refer to elements of $A$ as
  \emph{constructors} and say $B_a$ is the \emph{arity} of the
  constructor $a \in A$.
\end{defn}

\begin{defn}
  Given a polynomial $(B_a)_{a \in A}$, an \emph{algebra} is a set $X$
  together with a function $s \colon \sum_{a \in A} X^{B_a} \to X$. We
  refer to such $s$ as an \emph{algebra structure} on $X$.

  If $(X, s)$ and $(Y, t)$ are algebras, we say a \emph{homomorphism}
  is a function $h : X \to Y$ such that for all $a \in A$ and $f :
  B_a \to X$ we have $h(s(a, f)) = t(a, h \circ f)$.
\end{defn}

\begin{rmk}
  Although we assumed $X$ and $Y$ are sets in the definition above, we
  can also define algebra structures and homomorphisms for classes by
  replacing functions with class functions.
\end{rmk}

\begin{defn}
  Given a polynomial $(B_a)_{a \in A}$, an \emph{image preserving
    equation} over $(B_a)_{a \in A}$ consists of a set $V$, and
  $a, b \in A$ together with $l \colon B_a \to V$ and
  $r \colon B_{b} \to V$ such that the image of $l$ is equal to the
  image of $r$.

  A \emph{family of image preserving equations} consists of a set $E$
  together with a family of image preserving equations $(V_e, a_e,
  b_e, l_e, r_e)_{e \in E}$.
\end{defn}

\begin{defn}
  Given a polynomial $(B_a)_{a \in A}$ and a family of image
  preserving equations, $(V_e, a_e, b_e, l_e, r_e)_{e \in E}$, an
  \emph{algebra} is an algebra $(X, s)$ for the polynomial
  $(B_a)_{a \in A}$ such that for every $e \in E$ and every function
  $h \colon V_e \to X$ we have
  $s(a_e, h \circ l_e) = s(b_e, h \circ r_e)$.
\end{defn}

\begin{example}
  Suppose we are given a set $B$. We consider the polynomial with two
  constructors of arity $0$ and $B$. We consider the set of image
  preserving equations with set of variables $B$, and
  $l, r \colon B \to B$ defined by $l = 1_B$ and $r = \pi$ for each
  permutation $\pi \in \sym(B)$. The initial algebra is then the set
  of \emph{unordered trees of arity $B$}. In particular, we can take
  $B = \omega$ to get unordered countably branching trees.
\end{example}

\begin{example}
  Given any polynomial $(B_a)_{a \in A}$ and set $V$ we can consider the set of
  \emph{all} image preserving equations. We will see in the next
  section how this allows us to recover hereditarily countable sets
  and more generally hereditarily small sets.
\end{example}

\begin{example}
  \label{ex:nonexblass}
  We will be able to deduce from the main theorem that Blass' example
  of a $QW$-type that cannot be constructed in $\zf$ cannot be viewed
  as an image preserving $QW$-type. However, for illumination we will
  give a more intuitive direct reason why it does not satisfy the
  definition. In Blass' example, the initial algebra is expected to
  behave like the collection of all ordinals of countable cofinality. In
  particular there is an operation $\sup$ which takes a sequence
  $(\alpha_n)_{n < \omega}$ and is expected to behave like the
  supremum of the countable sequence of ordinals
  $(\alpha_n)_{n < \omega}$. In particular we should identify
  $\sup((\alpha_n)_{n < \omega})$ and $\sup((\beta_n)_{n < \omega})$
  whenever $\alpha_n$ is cofinal in $\beta_n$ and vice versa. However,
  this is much weaker than $\alpha_n$ and $\beta_n$ containing exactly
  the same elements (possibly in a different order).
\end{example}

\section{Hereditarily small sets}
\label{sec:hered-small-sets}

As part of \cite{jechheredctbl} Jech showed how to construct the set
of all hereditarily countable sets in $\zf$. This was later
generalised by Holmes \cite{holmesheredsmall} who showed that for any
set $B$ we can construct a set containing all sets hereditarily with
cardinality less than or equal to $B$. We in fact give a very slight
further generalisation of Holmes' result by instead considering
families of sets $(B_a)_{a \in A}$. In order to derive this from the
main theorem \ref{thm:iptypesexist}, we first need to clarify the
relation between image preserving equations and the class of
hereditarily small sets, which is the topic of this section.

\begin{defn}
  Given a family of sets $(B_a)_{a \in A}$ we say a set $X$ is
  \emph{small relative to $(B_a)_{a \in A}$} if for some
  $a \in A$ there exists a surjection $B_a \twoheadrightarrow X$. We
  say $X$ is \emph{hereditarily small relative to $(B_a)_{a \in A}$}
  if $X$ is small relative to $(B_a)_{a \in A}$ and all of its
  elements are hereditarily small relative to $(B_a)_{a \in A}$.
  We write $H((B_a)_{a \in A})$ for the class of hereditarily small
  sets.

  We view $H((B_a)_{a \in A})$ as an algebra on the polynomial
  $(B_a)_{a \in A}$ as follows. Given $a \in A$ and
  $f : B_a \to H((B_a)_{a \in A})$, we define $s(a, f)$ to be
  $\{ f(b) \;|\; b \in B_a \}$.
\end{defn}

\begin{lemma}
  \label{lem:heredsmsatip}
  $H((B_a)_{a \in A})$ satisfies all image preserving equations
  relative to the polynomial $(B_a)_{a \in A}$.
\end{lemma}

\begin{proof}
  Suppose we are given a set $S$ of variables together with $l : B_a
  \to S$ and $r : B_c \to S$ and a map $h : S \to
  H((B_a)_{a \in A})$. We calculate as follows.
  \begin{align*}
    s(h \circ l) &= \{ h(l(b)) \;|\; b \in B_a \} \\
                 &= \{ h(x) \;|\; x \in \im(l) \} \\
                 &= \{ h(x) \;|\; x \in \im(r) \} \\
                 &= s(h \circ r)
  \end{align*}
\end{proof}

We now fix a set of variables $S := \sum_{a \in A} \powset(B_a)$.

\begin{lemma}
  \label{lem:outofheredsm}
  Suppose $(X, t)$ is an algebra for $(B_a)_{a \in A}$ that satisfies
  all image preserving equations for the set of variables $S$
  above. Then there is a unique homomorphism
  $h : H((B_a)_{a \in A}) \to X$.
\end{lemma}

\begin{proof}
  In order for $h$ to be a homomorphism, we need it to satisfy the
  following equation whenever $a \in A$ and $f : B_a \to H((B_a)_{a
    \in A})$.
  \begin{equation*}
    h(\{ f(b) \;|\; b \in B_a \}) = t(a, h \circ f)
  \end{equation*}

  This defines a unique class function by $\in$-induction as long as
  the equation above is well defined. That is, given $g : B_c \to H((B_a)_{a
    \in A})$ such that $\{ f(b) \;|\; b \in B_a \} = \{ g(b) \;|\; b
  \in B_c \}$ we need $t(a, h \circ f) = t(c, h \circ g)$.

  Note that the lemma is trivial if $B_a$ is inhabited for each $a$,
  since then $H((B_a)_{a \in A})$ is empty. Hence we may assume $B_a$
  is empty for some $a$ and so $X$ has a canonical element $x_0$. Now
  fix $a, c \in A$ and $f$ and $g$ as above. By induction on rank, we
  may assume $h(y)$ is already uniquely defined for
  $y \in \{ f(b) \;|\; b \in B_a \}$. We define a map $k : S \to X$ as
  follows.
  \begin{equation*}
    k(a', C) =
    \begin{cases}
      h(y) & a' = a, C = f^{-1}(y) \text{ for unique } y \in
      \{ f(b) \;|\; b \in B_a \} \\
      x_0 & \text{otherwise}
    \end{cases}
  \end{equation*}

  We next define $l : B_a \to S$ and $r : B_c \to S$ as follows.
  \begin{align*}
    l(b) &:= (a, \{ b' \in B_a \;|\; f(b') = f(b) \}) \\
    r(b) &:= (a, \{ b' \in B_a \;|\; f(b') = g(b) \})
  \end{align*}

  Finally, it is straightforward to verify that $l$ and $r$ have the
  same image in $S$, and that $k \circ l = h \circ f$
  and $k \circ r = h \circ g$. Hence it follows from the image
  preserving equation $t(a, k \circ l) = t(c, k \circ r)$ that
  $t(a, h \circ f) = t(c, h \circ g)$ as required.
\end{proof}

\begin{thm}
  If $H((B_a)_{a \in A})$ is a set, then it is the $QW$-type for the
  polynomial $(B_a)_{a \in A}$ with all image preserving equations for
  $S$.
\end{thm}

\begin{proof}
  This follows directly from lemmas \ref{lem:heredsmsatip} and
  \ref{lem:outofheredsm}.
\end{proof}
  
\begin{thm}
  If the $QW$-type for the polynomial $(B_a)_{a \in A}$ with all image
  preserving equations for $S$ exists, then $H((B_a)_{a \in A})$ is a
  set.
\end{thm}

\begin{proof}
  Let $(X, t)$ be the $QW$-type.

  We first define a homomorphism from $X$ to $H((B_a)_{a \in A})$. The
  essential idea is to use the fact that $X$ is an initial
  algebra. However, formally speaking we can only apply initiality of
  $X$ once we know that $H((B_a)_{a \in A})$ is an algebra and in
  particular a set, which we don't have a priori. Hence we use a trick
  of approximating $H((B_a)_{a \in A})$ by a sequence of sets.
  
  For each ordinal $\alpha$ we can define an algebra
  structure on the set $(H((B_a)_{a \in A}) \cap V_\alpha) + 1$ that
  agrees with the structure $s$ on $H((B_a)_{a \in A})$ when $s(a, f)$
  belongs to $H((B_a)_{a \in A}) \cap V_\alpha$ and otherwise sends
  $(a, f)$ to the $1$ component. Note furthermore that this
  operation satisfies all image preserving equations by lemma
  \ref{lem:heredsmsatip}. Hence there is a unique homomorphism
  $h_\alpha : X \to (H((B_a)_{a \in A}) \cap V_\alpha) + 1$. Say that
  $x \in X$ is \emph{defined at $\alpha$} if
  $h(x) \in H((B_a)_{a \in A}) \cap V_\alpha$ and undefined
  otherwise. Note that for $\beta \geq \alpha$ the canonical
  restriction map
  $(H((B_a)_{a \in A}) \cap V_\beta) + 1 \to (H((B_a)_{a \in A}) \cap
  V_\alpha) + 1$ is a homomorphism. It follows that if $x$ is defined
  at $\alpha$, then it is also defined at $\beta$ and
  $h_\beta(x) = h_\alpha(x)$. By induction every element $x$ is
  defined at $\alpha$ for some ordinal $\alpha$. Hence this defines a
  unique homomorphism $X \to H((B_a)_{a \in A})$. The homomorphism has
  an inverse by lemma \ref{lem:outofheredsm}. Hence $H((B_a)_{a \in
    A})$ is in bijection with a set, and so a set itself.
\end{proof}

\section{Some useful propositions}
\label{sec:some-usef-prop}

We recall some basic classical set theory that will be useful for the
construction of image preserving $QW$-types.
We fill in some of the details, with the remainder left as an
exercise for the reader.

\begin{prop}
  For any ordinals $0 < \alpha < \beta$, there is a canonical
  surjection $\beta \twoheadrightarrow \alpha$.

  If there is a surjection $X \twoheadrightarrow \beta$ for some set
  $X$, then there is also a surjection $X \twoheadrightarrow \alpha$.
\end{prop}

\begin{prop}
  For any well ordered set $(X, <)$ (and in particular for sets of
  ordinals ordered by $\in$), there is a unique ordinal $\beta$ with a
  unique order isomorphism $(X, <) \cong (\beta, \in)$. We refer to
  $\beta$ as the \emph{order type} of $(X, <)$.
\end{prop}

\begin{prop}
  For any family of sets $(X_i)_{i \in I}$, there is an ordinal
  $\aleph((X_i)_{i \in I})$ which is the smallest for which there is
  no surjection $X_i \twoheadrightarrow \aleph((X_i)_{i \in I})$ for
  any $i \in I$. It is precisely the set of all ordinals $\alpha$ for
  which there is a surjection $X_i \twoheadrightarrow \alpha$ for some
  $i \in I$.
\end{prop}

\begin{proof}
  Note that whenever $X_i \twoheadrightarrow \alpha$, there is an
  equivalence relation $\sim$ on $X$, and a well ordering $<$ on
  $X/{\sim}$ such that the order type of $(X/{\sim}, <)$ is
  $\alpha$. However there is clearly a set of such well orders by
  power set, and so there is a set of all such ordinals
  $\alpha$. Since this is a downwards closed set of ordinals, it is an
  ordinal itself. Since the set cannot contain itself, it is the
  least ordinal for which there is no surjection from $X_i$ for any
  $i$.
\end{proof}

\begin{prop}
  \label{prop:cardinalseqs}
  If $\kappa$ is a cardinal number (i.e. an ordinal that is not in
  bijection with any lower ordinal), then one can define surjections
  \begin{enumerate}
  \item \label{sqrt} $\kappa \twoheadrightarrow \kappa \times \kappa$
  \item $\kappa \twoheadrightarrow \kappa^n$ for any $n < \omega$
  \item \label{timesomega}$\kappa \twoheadrightarrow \omega \times
    \kappa$
  \item $\kappa \twoheadrightarrow \sum_{n < \omega} \kappa^n$
  \end{enumerate}
\end{prop}

\begin{proof}
  For \ref{sqrt}, see e.g. \cite[Theorem I.11.30]{kunenfoundations}.

  For \ref{timesomega}, suppose we are given a bijective pairing
  function $(-,-) \colon \omega \times \omega \to \omega$. Any ordinal
  $\alpha$ can be written uniquely as $\alpha = \lambda + (m, n)$
  where $\lambda$ is a limit ordinal and $m, n \in \omega$. We then
  decode this as the pair $(m, \lambda + n)$, which clearly gives a
  bijection.

  Deriving the other parts from these two is straightforward.
\end{proof}

\section{The construction of image preserving $QW$-types}
\label{sec:proof}

We now construct image preserving $QW$-types in $\sets$. This is
based on Jech's construction of the set of hereditarily countable sets
\cite{jechheredctbl}.

\begin{defn}
  We will define a class function $Q$ from ordinals to sets by
  recursion on ordinals.

  We define $Q(0)$ to be $\emptyset$ and for limit ordinals $\lambda$,
  we define $Q(\lambda)$ to be $\bigcup_{\alpha < \lambda}
  Q(\alpha)$.

  We define $Q(\alpha + 1)$ as follows. Let $X$ be the set of pairs
  $(a, f)$ where $a \in A$ and $f : B_a \to Q(\alpha)$. We
  then take $\sim$ to be the equivalence relation on $X$ generated by
  identifying $(a_e, t \circ l_e)$ and $(b_e, t \circ r_e)$ whenever
  $t \colon V_e \to Q(\alpha)$ for $e \in E$ and  we define
  $Q(\alpha + 1)$ to be $X/{\sim}$.
\end{defn}

\begin{rmk}
  For $\beta \leq \alpha$ we have $Q(\beta) \subseteq Q(\alpha)$, by
  exploiting the fact that functions are implemented as relations not
  including explicit reference to their codomain, and noting that for
  $f : B_a \to Q(\alpha)$ and $g : B_{b} \to Q(\alpha)$ such that
  $(a, f) \sim (b, g)$, $f$ factors through the inclusion $Q(\beta)
  \subseteq Q(\alpha)$ if and only if $g$ does.
\end{rmk}

We now give a series of definitions and lemmas that apply at any
stage $\alpha \in \ords$.

\begin{defn}
  Note that we only identify $(a, f)$ and $(b, g)$ when $f$ and $g$
  have the same image in $Q(\alpha)$. Hence we have a well defined
  image function $\im \colon Q(\alpha) \to \powset(Q(\alpha))$, such
  that whenever $x = [(a, f)]$, $\im(x)$ is the image of $f$ in
  $Q(\alpha)$.
\end{defn}

\begin{defn}
  Given an element $x$ of $Q(\alpha)$ of the form $[(a, f)]$, we
  defined the \emph{rank} of $x$, $\rk(x)$ to be the smallest ordinal
  $\beta$ such that $f$ factors through the inclusion
  $Q(\beta) \subseteq Q(\alpha)$. To check this is a well
  defined, note that it depends only on the image of $f$.
\end{defn}

Note that $\rk(x) + 1$ is the smallest ordinal $\beta$ such that
$x \in Q(\beta)$.

\begin{defn}
  Given a set $X \subseteq Q(\alpha)$, we define the union $\cup X$
  by
  \begin{equation*}
    \cup X := \bigcup_{x \in X} \im(x)
  \end{equation*}

  We define the \emph{transitive closure} of $x \in Q(\alpha)$,
  $\tc(x)$ by
  \begin{equation*}
    \tc(x) := \bigcup_{1 \leq n < \omega} {\cup}^n \{x\}
  \end{equation*}
\end{defn}

\begin{lemma}
  \label{lem:rktc}
  For all $x \in Q(\alpha)$, we have the following equation.
  \begin{equation*}
    \rk(x) = \{ \rk(y) \;|\; y \in \tc(x) \}
  \end{equation*}
\end{lemma}

\begin{proof}
  It is clear that whenever $y \in \tc(x)$ we must have
  $\rk(y) < \rk(x)$ since this is the case for any $n < \omega$ and
  any $y \in \cup^n \{x\}$ by induction on $n$.

  It remains to show that for any $\beta < \rk(x)$, we have
  $\beta = \rk(y)$ for some $y \in \tc(x)$. By the definition of rank,
  $\im(x)$ cannot be contained in $Q(\beta)$. Hence we must have
  $x = [(a, f)]$ and $b \in B_a$ such that $f b \notin Q(\beta)$. For
  this $b$ we have $\beta \leq \rk(f b)$. If $\beta = \rk(f b)$, then
  $f b \in \cup\{x\} \subseteq \tc(x)$ and so $\beta$ is as
  required. Otherwise, $\beta < \rk(f b)$ and so by induction on $\rk(x)$
  we may assume $\beta = \rk(y)$ for some $y \in \tc(f b)$. However,
  $\tc(f b) \subseteq \tc(x)$, so $y \in \tc(x)$ and $\beta$ is again
  as required.
\end{proof}

\begin{defn}
  \label{defn:rns}
  For $x \in Q(\alpha)$, we write $R_n(x)$ for the set $\{\rk(z) \;|\; z
  \in \cup^n\{x\}\}$.
\end{defn}

\begin{lemma}
  \label{lem:rkunionrns}
  For all $x \in Q(\alpha)$,
  \begin{equation*}
    \rk(x) = \bigcup_{1 \leq n < \omega} R_n(x)
  \end{equation*}
\end{lemma}

\begin{proof}
  By lemma \ref{lem:rktc}.
\end{proof}

\begin{defn}
  We define $\kappa$ to be $\aleph((B_a)_{a \in A})$. We define
  $\kappa^+$ to be the smallest non zero ordinal for which there is no
  surjection $\kappa \twoheadrightarrow \kappa^+$.
\end{defn}

\begin{lemma}
  \label{lem:rnsurjs}
  We define for each $x \in Q(\alpha)$ and each $1 \leq n < \omega$, a
  surjection
  $F_{x, n} \colon \kappa^n \twoheadrightarrow R_n(x) \cup
  \{\emptyset\}$.
\end{lemma}

\begin{proof}
  We first consider the case $n = 1$. Suppose that $x = [(a,
  f)]$. Note that $R_1(x) := \{\rk(f b) \;|\; b \in B_a \}$ is a set
  of ordinals, and so it has an order type $\beta \in \ords$, and in
  particular we have a unique order isomorphism with $\beta$, say
  $\theta \colon \beta \stackrel{\cong}{\to} R_1(x)$. Furthermore, by
  definition, there is clearly a surjection from $B_a$ to $R_1(x)$. It
  follows that $\beta < \kappa$. Hence we can define a canonical
  surjection
  $F_{x, 1} \colon \kappa \twoheadrightarrow R_1(x) \cup
  \{\emptyset\}$ as follows.
  \begin{equation*}
    F_{x, 1}(\alpha) :=
    \begin{cases}
      \theta(\alpha) & \alpha < \beta \\
      \emptyset & \text{otherwise}
    \end{cases}
  \end{equation*}

  Now suppose $n = m + 1$. We fix $m$ ordinals less than $\kappa$, say
  $\beta_1, \ldots, \beta_m$ and consider the set $Y$ below.
  \begin{equation*}
    Y := \{ F_{f b, m}(\beta_1, \ldots, \beta_m) \;|\; b
    \in B_a \}
  \end{equation*}
  This is again a set of ordinals with a surjection from $B_a$ for
  some $a \in A$, and so as before, we have a canonical surjection
  $G \colon \kappa \twoheadrightarrow Y \cup \{ \emptyset \}$. We take
  $F_{x, n}( \beta_1, \ldots, \beta_m, \beta_{m + 1} )$ to be
  $G(\beta_{m + 1})$. We now simultaneously check that $F_{x, n}$ has
  the correct codomain and is surjective.
  \begin{align*}
    \im(F_{x, n})
    &= \bigcup_{\beta_1, \ldots, \beta_m < \kappa} (\{ F_{fb,
      m}( \beta_1, \ldots, \beta_m) \;|\; b \in B_a \} \cup \{
      \emptyset \}) \\
    &= \bigcup_{b \in B_a} \{ F_{fb,
      m}( \beta_1, \ldots, \beta_m ) \;|\;
      \beta_1, \ldots, \beta_m < \kappa \} \; \cup \; \{\emptyset \} \\
    &= \bigcup_{b \in B_a} \{ \rk(z) \;|\; z \in \cup^m\{ f b \} \}
      \; \cup \; \{ \emptyset \} \\
    &= R_n(x) \; \cup \; \{ \emptyset \}
  \end{align*}
\end{proof}

\begin{lemma}
  \label{lem:rkbound}
  For any $x \in Q(\alpha)$ we have $\rk(x) < \kappa^+$.
\end{lemma}

\begin{proof}
  First note that this is clear when $\rk(x) = 0$. Hence we may assume
  for the rest of the proof $\rk(x) > 0$.
  
  By the definition of $\kappa^+$, it suffices to define a surjection
  $\kappa \twoheadrightarrow \rk(x)$. By proposition
  \ref{prop:cardinalseqs} it suffices to define a surjection
  $\sum_{1 \leq n < \omega} \kappa^n \twoheadrightarrow
  \rk(x)$. However, by lemma \ref{lem:rkunionrns} we can express
  $\rk(x)$ as $\bigcup_{1 \leq n < \omega} R_n(x)$. Since
  $\rk(x) > 0$, this is the same as
  $\bigcup_{1 \leq n < \omega} (R_n(x) \cup \{ \emptyset \})$, and so
  we can just combine the surjections defined in lemma
  \ref{lem:rnsurjs}.
\end{proof}

\begin{thm}
  \label{thm:iptypesexist}
  All image preserving $QW$-types exist in $\sets$.
\end{thm}

\begin{proof}
  We show that $Q(\kappa^+)$ is an initial algebra. We first need to
  show how to define an algebra structure. Suppose we are given $a \in
  A$ and a map $f \colon B_a \to Q(\kappa^+)$. Then $[(a, f)]$ is an
  element of $Q(\kappa^+ + 1)$. By lemma \ref{lem:rkbound} we have
  $\rk([(a, f)]) < \kappa^+$, and so $f$ factors through $Q(\beta)$ for
  some $\beta < \kappa^+$. We can then take $\sup(a, f)$ to be
  $[(a, f)] \in Q(\beta + 1)$.

  We check that this structure respects the equations. Suppose that we
  are given $g \colon V_e \to Q(\kappa^+)$. Note that $g \circ l_e$
  and $g \circ r_e$ have the same image, and so
  $\rk([(a_e, g \circ l_e)]) = \rk([(b_e, g \circ r_e)])$. Hence
  $[(a_e, g \circ l_e)]$ and $[(b_e, g \circ r_e)]$ must
  have been first added at the same stage, $\alpha + 1$. We can now
  see that they were identified in the definition of $Q(\alpha + 1)$.

  Finally, it is clear that for any other algebra structure, we can
  define a unique homomorphism out of $Q(\kappa^+)$ by
  recursion on ordinals.
\end{proof}

\section{The ranks of unordered countably branching trees}
\label{sec:ranks-unord-count}

It follows from Jech's construction that every hereditarily countable
set has rank less than $\omega_2$. In \cite{jechheredctbl} Jech also showed
a converse statement when $\omega_1$ is singular: in this case
there is a hereditarily countable set of rank $\alpha$ for every
$\alpha < \omega_2$. We will show the analogous result for unordered
countably branching trees. We first recall that Jech proved the
following lemma as part of the proof of \cite[Theorem
2]{jechheredctbl}.

\begin{lemma}
  \label{lem:choosecofsets}
  Suppose that $\omega_1$ is singular and let $\alpha < \omega_2$. For
  each fixed surjection $f : \omega_1 \twoheadrightarrow \alpha$ and
  choice of countable cofinal sequence in $\omega_1$, we can construct
  for each limit ordinal $\lambda \leq \alpha$ a choice of countable
  cofinal subset $Y_\lambda \subseteq \lambda$.
\end{lemma}

We can now show the theorem.

\begin{thm}
  \label{thm:ucbtrank}
  Suppose that $\omega_1$ is singular. Then for every $\alpha <
  \omega_2$ there is an unordered countably branching tree of rank
  $\alpha$.
\end{thm}

\begin{proof}
  Since $\alpha < \omega_2$ and $\omega_1$ is singular there exists a
  choice of sets $Y_\lambda$ as in lemma \ref{lem:choosecofsets}
  for limit ordinals $\lambda \leq \alpha$.

  We construct an unordered countably branching tree $t_\beta$ for
  each $\beta \leq \alpha$ by induction. We define
  $t_0 := \mathtt{leaf}$. For successor ordinals we define
  $t_{\beta + 1} := \node(\lambda n.t_\beta)$. This just
  leaves the case of $t_\lambda$ for non zero limit ordinals
  $\lambda$. Since $Y_\lambda$ is countable there exists a surjection
  $g : \omega \twoheadrightarrow Y_\lambda$. Since $Y_\lambda$ is
  cofinal in a limit ordinal it is not finite. Hence there exists a
  bijection $h : \omega \stackrel{\cong}{\to} Y_\lambda$, which we
  can construct from a choice of surjection $g$ as follows:
  \begin{align*}
    h(0) &= g(0) \\
    h(n + 1) &=
               \begin{cases}
                 g(n + 1) & g(n + 1) \neq h(k) \text{ for all } k \leq
                 n
                 \\
                 \gamma & \text{otherwise},
                 \gamma \text{ least s.t. } \gamma \neq h(k)\text{ for
                   all } k \leq n
               \end{cases}
  \end{align*}

  We show how to define $t_\lambda$ from a choice of bijection $h$,
  and then check that it is independent of the particular choice and
  so well defined.

  We define $t_\lambda := \node(\lambda n.t_{h(n)})$. Now
  suppose that we are given two such bijections
  $h, h' : \omega \stackrel{\cong}{\to} Y_\lambda$. Then
  $\pi := h^{-1} \circ h'$ is a permutation of $\omega$ and
  $h' = h \circ \pi$. Hence $\node(\lambda n.t_{h'(n)})$ and
  $\node(\lambda n.t_{h(n)})$ are identified by one of the
  defining equations of unordered countably branching trees, ensuring
  that $t_\lambda$ is independent of the choice of $h$, as
  required. It is clear that $\rk(t_\beta) = \beta$ for all $\beta
  \leq \alpha$.  
\end{proof}

\begin{cor}
  If $\omega_1$ is singular, then the unique homomorphism
  $h : T \to T_{\sym}$ from countably branching trees to unordered
  countably branching trees is not surjective.
\end{cor}

\begin{proof}
  Note that $h$ preserves rank. Hence by theorem \ref{thm:ucbtrank} it
  suffices to show that the rank of every (ordered) countably
  branching tree is less than $\omega_1$.

  Fix a surjective function
  $s : \omega \twoheadrightarrow \omega + (\omega \times \omega)$. We
  construct for each ordered tree $t$ a countable enumeration
  $f(t) : \omega \to T$ whose image is the union of the set of proper
  subtrees of $t$ and $\{\mathtt{leaf}\}$. We define
  $f(\mathtt{leaf})(n) := \mathtt{leaf}$. Now suppose we are given a
  tree of the form $t = \node(g)$. By composing with $s$, it suffices
  for us to define two functions $f_0 : \omega \to T$ and
  $f_1 : \omega \times \omega \to T$ that jointly enumerate all of the
  proper subtrees of $t$ (including $\mathtt{leaf}$). We define
  $f_0(n) := g(n)$ and $f_1(n, m) := f(g(n))(m)$.

  By lemma \ref{lem:rktc} the rank of each $t$ is equal to the set of
  ranks of proper subtrees. However, by the above enumeration this is
  a countable ordinal and so less than $\omega_1$, as required.
\end{proof}

\section{Conclusion}
\label{sec:conclusion}

We have shown that every image preserving $QW$-type exists in $\sets$
under $\zf$ without any additional assumptions. Although the proof is
somewhat elaborate, we can see from the results of section
\ref{sec:ranks-unord-count} that some complication is necessary. The
na\"{i}ve construction of defining an equivalence relation on the
$W$-type of the underlying polynomial and quotienting will not work in
general. As we saw, both hereditarily countable sets and unordered
countably branching trees provide counterexamples when $\omega_1$ is
singular.

\subsection{Generalisations and limitations of this method}
\label{sec:gener-limit-this}

The proof in section \ref{sec:proof} made use of the fact that we have
a well defined image operation sending each element of the form
$(a, f)$ to the image of $f$. This may not be strictly necessary for
the proof, since the key lemma \ref{lem:rnsurjs} only used the sets
$R_n$ in definition \ref{defn:rns} which may be definable in
situations without a well defined image operator. So there could be a
more general version of the theorem that uses the same key
idea. However, it is unclear if there are any interesting examples of
$QW$-type that can be constructed by the same general method, but
whose existence isn't already a corollary of theorem
\ref{thm:iptypesexist}. We therefore leave it as an open problem both
to generalise the theorem and to find interesting examples making
essential use of the generalisation.

On the other hand there are some examples where the argument here
cannot possibly apply. As already discussed in example
\ref{ex:nonexblass}, Blass' counterexample cannot be
constructed by this method, since it does not provably exist under
$\zf$. However, there is also a much simpler example of a $QW$-type
that exists, provably in $\zf$, but does not seem to be covered by the
argument in this paper, namely free monoids. We consider in particular
the free monoid on a set with four elements $\{a, b, c, d\}$. As a
special case of associativity we have $a(b(cd)) = (ab)(cd)$. However,
it is unclear how to define $R_1$ for this term. Since $\rk(a) = 0$
and $\rk(b(cd)) = 2$, we expect $R_1(a(b(cd))) = \{0, 2\}$. Since
$\rk(ab) = \rk(cd) = 1$, we also expect $R_a((ab)(cd)) = \{1 ,
1\}$. Since $R_n$ play an essential role in lemma \ref{lem:rnsurjs} it
seems that it is impossible to construct free monoids by this method,
even though they do provably exist in $\sets$ under $\zf$.

Another class of examples of $QW$-types that exist under $\zf$, but are not
covered directly by the methods of section \ref{sec:proof} are $W$-types with
reductions. By their nature they identify two trees of different rank,
so we don't expect to have a well defined of rank without the
observation in \cite{swanwtypered} that in the presence of the law of
excluded middle they can be viewed as an ordinary $W$-types
of normal forms.

\subsection{Related Open Problems}
\label{sec:relat-open-probl}

The limitations of the result discussed above naturally lead to the
following question.
\begin{openprob}
  Is there a general construction of $QW$-types in $\sets$ under $\zf$
  that naturally encompasses all examples known to exist in this
  setting?
\end{openprob}

Our proof of lemma \ref{lem:rnsurjs} makes essential use of the fact
that any set of ordinals is order isomorphic to an ordinal, which in
turn uses classical logic. This leaves open the case of intuitionistic
logic, specifically the following problem.
\begin{openprob}
  Show all image preserving $QW$-types exist in $\sets$ under the
  assumptions of $\izf$ or find a model of $\izf$ where they do not.
\end{openprob}

Note that the proof of theorem \ref{thm:ucbtrank} made essential use
of the fact that we have equations for all permutations $\pi$ of
$\omega$. It is unclear what happens when we allow some permutations
but not all, leading to the following question.
\begin{openprob}
  We define the $QW$-type of \emph{countably branching weakly
    unordered trees} to have the same underlying polynomial as
  countably branching trees and equations of the form
  $\node(\lambda n.v_n) = \node(\lambda n.v_{\pi(n)})$ where $\pi$ is
  a \emph{finitely supported} permutation of $\omega$. Is it provable
  in $\zf$ that the unique homomorphism from countably branching trees
  to weakly unordered countably branching trees is a surjection?
\end{openprob}

\bibliography{mybib}{}
\bibliographystyle{alpha}

\end{document}